\documentclass[12pt,a4paper]{article}
\usepackage{amssymb,amsmath,amsthm}

\usepackage{times}
\usepackage{url}
\usepackage{graphicx}
\usepackage{epstopdf}
\def\de{{\rm d}}

\newtheorem{theorem}{Theorem}[section]

\newtheorem{corollary}{Corollary}[section]
\newtheorem{remark}{Remark}[section]

\numberwithin{equation}{section}

\title{Change point estimation for the telegraph process observed at discrete times}

\begin{document}

\author{Alessandro De Gregorio,  \,\,Stefano M. Iacus\footnote{\textbf{email:} alessandro.degegorio@unimi.it, stefano.iacus@unimi.it}\\
{\it Dipartimento di Scienze Economiche, Aziendali e Statistiche}\\
{\it	Via Conservatorio 7,  20122 Milan - Italy}
}
\maketitle

\begin{abstract}
The telegraph process models
a random motion with finite velocity and it is usually proposed as 
an alternative to diffusion models.
The process
 describes the position of a particle moving
on the real line, alternatively with constant velocity $+ v$ or
$-v$. The changes of direction are governed by an homogeneous
Poisson process with rate $\lambda >0.$ 
In this paper, we consider a change point estimation problem for the rate of the underlying Poisson process by means of least squares method. The consistency and the rate of convergence for the change point estimator are obtained and its asymptotic distribution is derived. Applications to real data are also presented. 
\end{abstract}


\noindent

\newpage

\section{Introduction}
The telegraph process describes
a random motion with finite velocity and it is usually proposed as 
an alternative to classical diffusion models (see Goldstein, 1951 and Kac, 1974).
The process defines the position of a particle initially located at the origin of the real line and moving alternatively with constant velocity $+ v$ or
$-v$. The changes of direction are governed by an homogeneous
Poisson process with rate $\lambda
>0.$ The telegraph process or  {\it telegrapher's} process is defined as
\begin{equation}\label{1.1}
X(t)=  V(0)\int_0^t (-1)^{N(s)}\de s,\quad t>0,
\end{equation} 
where  $V(0)$ is the initial velocity taking values  $\pm v$ with equal 
probability and independently of the Poisson process  $\{N(t), t>0\}$.
Many authors analyzed probabilistic properties of the process over the years (see
for example Orsingher, 1990, 1995; Foong and Kanno, 1994; Stadje and Zacks, 2004; Zacks 2004).
Di Crescenzo and Pellerey (2002) proposed the geometric telegraph
process as a model to describe the dynamics of the price of risky
assets
where $X(t)$ replaces the standard Brownian motion of the original Black-Scholes-Merton model. Conversely to the  geometric Brownian motion,
given that $X(t)$ is of bounded variation, so is the geometric telegraph process. This seems a realistic way to model paths of assets in the financial markets.
Mazza and Rulliere (2004) linked the process \eqref{1.1} and the
ruin processes in the context of risk theory. Di Masi {\it et al.} (1994) proposed to model the volatility of financial markets in terms of the telegraph process. Ratanov (2004, 2005) proposed to model financial markets using a telegraph process with two intensities $\lambda_\pm$ and two velocities $v_\pm$. The telegraph process has also been used in  ecology to model population dynamics (see Holmes {\it et al.}, 1994) and the displacement of wild animals on the soil. In particular, this model is chosen because it preserves the property of animals to move at finite velocity and for a certain period along one direction (see e.g. Holmes, 1993, for an account).

For the telegraph process  $\{X(t), 0 \leq t \leq T\}$ observed at
equidistant discrete times $0=t_0<\ldots<t_n$, with
$t_i=i \Delta_n$, $i=0, \ldots, n$, $n\Delta_n=T$ and $\Delta_n\to0$ as $n\to\infty$,
De Gregorio and Iacus (2006) proposed pseudo-maximum likelihood and implicit moment based estimators for the rate $\lambda$ of the telegraph process. 
Under the additional condition $n\Delta_n\to\infty$ as $n\to\infty$, Iacus and Yoshida (2007) studied the asymptotic properties of  explicit moment type estimators and further propose a consistent, asymptotically gaussian and asymptotically efficient estimator based on the increments of the process.

In this paper we suppose that for a telegraph process occurs a switch of the rate from $\lambda_1$ to $\lambda_2$ at some time instant $\theta_0\in [0,T]$ and the interest is in the estimation of the change point $\theta_0$ and both $\lambda_1$ and $\lambda_2$.

The change point estimation theory has been employed widely by means of the likelihood function (see  Cs\"{o}rg\H{o} and Horv\'{a}th, 1997). Unfortunately, the likelihood function for the telegraph process is not known and the pseudo likelihood proposed in De Gregorio and Iacus (2006) is not easy to treat in this framework. We will then proceed using the  alternative method based on least squares proposed  in  Bai (1994, 1997) and used in different contexts by many authors including  Hsu (1977, 1979) for the i.i.d. case and Chen {\it et al.}  (2005) for the mixing case.
Our model is peculiar in itself for the properties of $X(t)$ described in the above, because $\lambda$ is a parameter related to the mean and the variance of the Poisson process and because the mesh $\Delta_n$ plays a role in  the definition of the rate of convergence of our estimators.

The paper is organized as follows. Section \ref{sec:model} describes the model, the observation scheme and the change point estimator. The consistency of change point estimator is discussed in Section \ref{consistency} and distributional results are presented in Section \ref{dist}. Finally, Section \ref{appl} contains an application of our method to real data: we will reanalyze the classical change point data sets of Dow-Jones weekly closing (see Hsu, 1977, 1979) and IBM stock prices (see Box and Jenkins, 1970).

\section{Observation scheme and least squares estimator}\label{sec:model}

We consider a telegraph process $\{X(t), 0 \leq t \leq T<\infty\}$ defined as in \eqref{1.1} and assume to observe its trajectory only in $n+1$ equidistant discrete times $0<t_1<\ldots<t_n,$ with $t_i=i\Delta_n$, $i=1,\ldots ,n$ and $n\Delta_n=T$. We assume that a rate of changes of direction shift occurs during the interval $[0,T]$ at an unknown time $\theta_0 = \tau T$, $\tau\in(0,1)$. Therefore the changes of direction are governed by an inhomogeneous Poisson process with parameter 
$\lambda(t)=\lambda_1\mathbf{1}_{\{t\leq \theta_0\}}+\lambda_2\mathbf{1}_{\{t> \theta_0\}}$
where the positive values $\lambda_1,\lambda_2$ and the change point $\tau$ (or $\theta_0$) are unknown and to be estimated given the observations $X(t_1)$, $X(t_2)$, \ldots, $X(t_n)$.
In order to simplify the formulas we use the following notation: $X(t_i) = X(i\Delta_n) = X_i$. The
asymptotic framework is the following: $\Delta_n \rightarrow 0$ and $n\Delta_n=T\to \infty$ as   $n\to \infty$.  

The telegraph process is not Markovian and, as mentioned in the Introduction,  it is not possible to derive the explicit likelihood function of the observations $X_i$'s, therefore we can not apply the statistical methods based on the likelihood function.
To work out our estimation problem, we shall follow the approach developed in Bai (1994), which involves least squares-type estimators. The same point of view has been applied by Chen {\it et al.}, 2005,  in a context of financial time series. For our model, the time increment $\Delta_n$ plays an active role in the study of the asymptotics of our estimators so the proofs, although in some cases along the lines of Bai (1994) require some technical, but crucial, adjustments.

In order to obtain our estimator we introduce some basic notations. Let $$Y_i = \frac{\mathbf{1}_{\{|\eta_i|<v\Delta_n\}}}{\Delta_n},\,i=1,...,n,$$
where $\eta_i=X_i-X_{i-1}$ is the increment between two consecutive observations. 
We indicate the mean value of $Y_i$ with $\gamma_m=\frac{1-e^{-\lambda_m\Delta_n}}{\Delta_n}=\lambda_m+o(1)$, $m=1, 2$. We observe that the random variables $Y_i$ are independent and identically distributed because depend by the increments $|\eta_i|$. Iacus and Yoshida (2007) proved that the estimators
$$\hat\gamma_n = \frac{1}{n} \sum_{i=1}^n Y_i \qquad 
\text{and}\qquad
\hat\lambda_n = -\frac{1}{\Delta_n}\log(1 - \hat\gamma_n \Delta_n)$$
are consistent, gaussian and asymptotically efficient estimators of $\gamma$ and $\lambda$ respectively. We will use these properties in the following without necessarily mentioning them.

We assume that the change occurs exactly at time $t_i=t_{k_0}=k_0\Delta_n=\theta_0$, therefore $k_0 = [n\tau]$, $\tau\in (0,1)$, where $[\,\cdot\,]$ represents the integer-valued function. The change point estimator is obtained as follows
\begin{eqnarray}
\hat k_0& =& \arg\min_k\left( \min_{\gamma_1, \gamma_2}
\left\{\sum_{i=1}^{k} (Y_i - \gamma_1)^2 + \sum_{i=k+1}^n (Y_i -
\gamma_2)^2 \right\}
\right)\notag\\
&=&\arg\min_k\left\{\sum_{i=1}^{k} (Y_i - \bar Y_k)^2+\sum_{i=k+1}^n (Y_i -\bar Y_{k+1})^2\right\},
\label{eq:ls}
\end{eqnarray}
where $$\min_{\gamma_1}\sum_{i=1}^{k} (Y_i - \gamma_1)^2
=\frac{1}{k}\sum_{i=1}^k Y_i=\bar Y_k ,$$
$$\min_{\gamma_2}\sum_{i=k+1}^{n} (Y_i - \gamma_2)^2 =\frac{1}{n-k}\sum_{i=k+1}^n Y_i=\bar Y_{n-k} .$$
 We indicate the sum of the squares of residuals in the following manner
\begin{equation}
U_k^2=\sum_{i=1}^{k} (Y_i - \bar Y_k)^2+\sum_{i=k+1}^n (Y_i -\bar Y_{k+1})^2,
\end{equation}
then

\begin{equation}\label{ls}
\hat k_0= \arg\min_k U_k^2,
\end{equation}
and $\hat \gamma_1=\bar Y_{\hat k_0}$, $\hat \gamma_2=\bar Y_{n-\hat k_0}$ are respectively the least squares estimators of $\gamma_1$ and $\gamma_2.$ This gives the two estimators  
\begin{equation}\label{estlam}
\hat\lambda_1=-\frac{1}{\Delta_n}\log(1-\hat\gamma_1\Delta_n) ,\qquad \hat\lambda_2=-\frac{1}{\Delta_n}\log(1-\hat\gamma_2\Delta_n).
\end{equation}
By setting $\bar Y_n=\frac{1}{n}\sum_{i=1}^n Y_i$, $S_n=\sum_{i=1}^n
Y_i$, 
simple algebra leads to
\begin{equation*}
\sum_{i=1}^n(Y_i-\bar Y)^2=U_k^2+nV_k^2,\quad1\leq k \leq n-1,
\end{equation*}
where
\begin{eqnarray}
V_k&=&\sqrt{\frac{k(n-k)}{n^2}}(\bar Y_{n-k}-\bar Y_k)=\frac{1}{\sqrt{k(n-k)}}S_n D_k,
\end{eqnarray}
and 
\begin{equation}
\label{DK}
D_k=\frac{k}{n}-\frac{S_k}{S_n}.
\end{equation}
Therefore, formula \eqref{ls} implies that 
\begin{equation}\label{ls2}
\hat k_0= \arg\max_k V_k^2=\arg\max_k |D_k|=\arg\max_k \sqrt{k(n-k)}|V_k|.
\end{equation}

Our first result concerns the asymptotic distribution of $\hat k_0$ under the null hypothesis that $\lambda_1=\lambda_2$. This permits us to test if a shift has taken place during the interval $[0,T]$. 

\begin{theorem}\label{H0}
Under $H_0$, i.e. $\lambda_2=\lambda_1=\lambda_0,$ we have that
for $\Delta_n\to 0,n\Delta_n\to \infty$ as $n\to \infty$ the
following result hods
\begin{equation}
\sqrt{\frac{n\Delta_n}{\lambda_0}}|D_k|\stackrel{d}{\to}
|B^0(t)|,
\end{equation}
where $\{B^0(t), 0 \leq t \leq 1\}$ is a Brownian bridge.
\end{theorem}
\begin{proof}
Let $\xi_i=Y_i-\gamma_0$, $\gamma_0=EY_i
=\frac{1-e^{-\lambda_0\Delta_n}}{\Delta_n}=\lambda_0+o(1)$. Then,
$E\xi_i=0$ and
$\sigma^2_{n}=Var(\xi_i)=\frac{(1-e^{-\lambda_0\Delta_n})e^{-\lambda_0\Delta_n}}{\Delta_n^2}.$
We introduce the following function
$$X_n(t)=\frac{1}{\sigma_{n} \sqrt{n}}\mathcal{S}_{\left[nt\right]}+(nt-\left[nt\right])\frac{1}{\sigma_n \sqrt{n}}\xi_{\left[nt\right]+1},\quad 0<t<1,$$
with $\mathcal{S}_n=\sum_{i=1}^n\xi_i.$ We note that
\begin{equation}\label{ugp}
\left|\frac{1}{\sigma_n\sqrt{n}}\sum_{i=1}^{[nt]}\xi_i-\frac{\sqrt{t}}{\sigma_n\sqrt{[nt]}}\sum_{i=1}^{[nt]}\xi_i\right|\stackrel{p}{\to}0,
\end{equation}
and
\begin{eqnarray}\label{var}
Var
&\biggl(&\frac{\sqrt{t}}{\sigma_n\sqrt{[nt]}}\sum_{i=1}^{[nt]}\xi_i\biggr)=Var
\left(\frac{\sqrt{t}}{\sigma_n\sqrt{[nt]\Delta_n}}\sum_{i=1}^{[nt]}\Delta_n\xi_i\right)\\
&=&Var
\left(\frac{\sqrt{t}}{\sigma_n\sqrt{[nt]\Delta_n}}\sum_{i=1}^{[nt]}\frac{\mathbf{1}_{\{|\eta_i|<v\Delta_n\}}-(1-e^{-\lambda_0\Delta_n})}{\sqrt{\Delta_n}}\right)\notag\\
&=&t\notag
\end{eqnarray}
Since
$\left|\frac{\mathbf{1}_{\{|\eta_i|<v\Delta_n\}}-(1-e^{\lambda_0\Delta_n})}{\sqrt{\Delta_n}}\right|
<\frac{1}{\sqrt{\Delta_n}}$ the Lindeberg condition is true
\begin{equation}\label{Lin}
\sum_{i=1}^{[nt]}\frac{E\left\{\mathbf{1}_{\{\sqrt{\Delta_n}|\xi_i|\geq
\varepsilon
\sqrt{n\Delta_n}\sigma_n\}}\xi_i^2\right\}}{\Delta_n\sigma_n^2([nt])^2}\to
0.
\end{equation}
Then from \eqref{ugp}, \eqref{var} and \eqref{Lin} we can conclude that

\begin{equation}\label{inv}
\frac{1}{\sigma_{n}
\sqrt{n}}\mathcal{S}_{\left[nt\right]}\stackrel{d}{\to}N(0,t).
\end{equation}

Now, by applying Donsker's theorem (invariance principle) we are able to write that
$$X_n(t)\stackrel{d}{\to}B(t)$$
$$\left\{X_n(t)-tX_n(1)\right\}\stackrel{d}{\to}B^0(t), $$
with $B(t)$ and $B^0(t)$ representing respectively a standard
Brownian motion and a Brownian bridge. Let $k=[nt]$,
we can write
\begin{eqnarray*}
X_n(t)-tX_n(1)&=&\frac{1}{\sigma_n \sqrt{n}}\left[\mathcal{S}_k-\frac{k}{n}\mathcal{S}_n\right]+\frac{nt-\left[nt\right]}{\sigma_n\sqrt{n}}\xi_{\left[nt\right]+1}\\
&=&\frac{1}{\sigma_n
\sqrt{n}}\left[\sum_{i=1}^k(Y_i-\gamma_0)-\frac{k}{n}
\sum_{i=1}^n(Y_i-\gamma_0)\right]\\
&&+\frac{nt-\left[nt\right]}{\sigma_n\sqrt{n}}\xi_{\left[nt\right]+1}.
\end{eqnarray*}
We observe that
$$\sum_{i=1}^k(Y_i-\gamma_0)-\frac{k}{n} \sum_{i=1}^n(Y_i-\gamma_0)=-D_k\sum_{i=1}^nY_i ,$$
and consequently
\begin{equation}\label{teo2}
\sqrt{n\Delta_n}|D_k|\frac{\sum_{i=1}^nY_i
}{n\sqrt{\frac{(1-e^{\lambda_0\Delta_n})e^{\lambda_0\Delta_n}}{\Delta_n}}}=\left|X_n(t)-tX_n(1)-\frac{nt-\left[nt\right]}{\sigma_n\sqrt{n}}\xi_{\left[nt\right]+1}\right|.
\end{equation}
It is easy to see (by Chebyshev inequality) that
$$\sup_t\left|\frac{nt-\left[nt\right]}{\sigma_n\sqrt{n}}\xi_{\left[nt\right]+1}\right|\stackrel{p}{\to}0.$$
By the law of large number $\frac{1}{n}\sum_{i=1}^nY_i\stackrel{p}{\to}
\lambda_0,$ while $\sqrt{\frac{(1-e^{\lambda_0\Delta_n})e^{\lambda_0\Delta_n}}{\Delta_n}}\to \sqrt{\lambda_0}$. Therefore from \eqref{teo2} follows that
$$\sqrt{\frac{n\Delta_n}{\lambda_0}}|D_k|\stackrel{d}{\to} |B^0(t)|.$$
\end{proof}
\begin{corollary}
The same convergence result of the Theorem \ref{H0} follows when we consider 
\begin{equation*}
\frac{\sqrt{n\Delta_n}}{\sqrt{\tilde\lambda_0}}|D_k|,
\end{equation*}
where $\tilde\lambda_0$ is any consistent estimator for $\lambda_0$.
\end{corollary}

\begin{remark}
From Theorem \ref{H0} we derive immediately that for $\delta \in (0,1/2)$
\begin{equation}
\sqrt{\frac{n\Delta_n}{\lambda_0}}\sup_{\delta n\leq k \leq (1-\delta)n}|D_k|\stackrel{d}{\to}
\sup_{\delta\leq t \leq (1-\delta)}|B^0(t)|,
\end{equation}
\begin{equation}\label{ash0}
\sqrt{\frac{n\Delta_n}{\lambda_0}}\sup_{\delta n\leq k \leq (1-\delta)n}|V_k|\stackrel{d}{\to}
\sup_{\delta\leq t \leq (1-\delta)}(t(1-t))^{-1/2}|B^0(t)|.
\end{equation}

The last asymptotic results are useful to test if doesn't exist a change point. In particular it is possible to obtain the asymptotic critical values for the distribution \eqref{ash0}  by means of the same arguments used in  Cs\"{o}rg\H{o} and Horv\'{a}th (1997), pag. 25. 
\end{remark}

\section{The consistency properties of the estimator}\label{consistency}

We shall study the consistency and the rate of convergence of the change point estimator \eqref{ls2}. It is convenient to note that the rate of convergence is particularly important not only to describe how fast the estimator converges to the true value, but also to get the limiting distribution. The next Theorem represents our first result on the consistency.

\begin{theorem}\label{teo1}
The estimator $\hat\tau=\frac{\hat{k}_0}{n}$ satisfies
\begin{equation}\label{result1}
|\hat\tau-\tau|=(n\Delta_n)^{-1/2}(\gamma_2-\gamma_1)^{-1}O_p(\sqrt{\log n})
\end{equation}
\end{theorem}

\begin{proof}
By the same arguments of Bai (1994), Section 3 and by using the formulas (10)-(14) therein, we have that
\begin{equation}\label{boundin}
|\hat\tau-\tau|\leq  C_\tau (\gamma_2-\gamma_1)^{-1}\sup_k|V_k-EV_k|,
\end{equation}
where $C_\tau$
is a   constant depending only on $\tau$. Let
$Z_i=\frac{\mathbf{1}_{\{|\eta_i|<v\Delta_n\}}-(1-e^{\lambda\Delta_n})}{\sqrt{\Delta_n}}$,
given that
\begin{eqnarray*}
V_k-EV_k&=&\frac{1}{\sqrt{n\Delta_n}}\sqrt{\frac{k}{n}}\frac{1}{\sqrt{n-k}}\sum_{i=k+1}^n Z_i\\
&&+\frac{1}{\sqrt{n\Delta_n}}\sqrt{1-\frac{k}{n}}\frac{1}{\sqrt{k}}\sum_{i=1}^k Z_i
\end{eqnarray*}
we obtain that
\begin{equation}\label{intp}
|V_k-EV_k|\leq \frac{1}{\sqrt{n\Delta_n}}\left\{\bar{Z}_{n-k}+\bar{Z}_k\right\}.
\end{equation}
where $\bar Z_k = \frac{1}{\sqrt{k}} \sum_{i=1}^k Z_i$ and
 $\bar Z_{n-k} = \frac{1}{\sqrt{(n-k)}} \sum_{i=k+1}^n Z_i$.
By applying Haj\'ek-Renyi inequality for martingales we have that
\begin{eqnarray}
P\left\{\max_{1\leq k \leq n}\left|\frac{\sum_{i=1}^k Z_i}{c_k}\right|>\alpha \right\} &\leq& \frac{1}{\alpha^2}\sum_{k=1}^n\frac{E(Z_k ) ^2}{c_k^2}\notag\\
&=& \frac{  (1-e^{-\lambda\Delta_n})e^{-\lambda\Delta_n}}{\alpha^2\Delta_n} \sum_{k=1}^n\frac{1}{c_k^2}\notag\\
&\leq& \frac{\lambda\Delta_n+o(\Delta_n)}{\alpha^2\Delta_n} \sum_{k=1}^n\frac{1}{c_k^2}\notag\\
&=& \frac{\lambda+o(1)}{\alpha^2}
\sum_{k=1}^n\frac{1}{c_k^2}\label{hr}
\end{eqnarray}
Choosing $c_k=\sqrt{k}$ and observing that $\sum_{k=1}^n k^{-1} \leq  C \log n$, for some $C>0$ (see e.g. Bai, 1994),  we have that
\begin{equation}\label{ratecon}
\max_{1\leq k \leq n}\frac{1}{\sqrt{k}}\sum_{i=1}^k Z_i=O_p\left(\sqrt{\log n}\right).
\end{equation}
Then from the relationships \eqref{intp} and \eqref{ratecon} we obtain the result \eqref{result1}.
\end{proof}

\begin{remark}
By means of the law of iterated logarithm we obtain immediately the following rate of convergence which improve the previous result. 
We have that
\begin{equation}\label{result2}
|\hat\tau-\tau|=(n\Delta_n)^{-1/2}(\gamma_2-\gamma_1)^{-1}O_p(\sqrt{\log\log n}).
\end{equation}
\end{remark}

\begin{remark}
Theorem \ref{teo1} implies that, under the additional hypothesis $\Delta_n = O(n^\varepsilon)$, $\varepsilon\in(-1,0)$ we have also consistency, i.e. $(n\Delta_n)^\beta(\hat\tau-\tau)\to0$ in probability for any $\beta\in(0,1/2)$.
\end{remark}
We are able to improve the rate of convergence of $\hat\tau$.

\begin{theorem}\label{consim}
We have the following result
\begin{equation} \label{imrate}
\hat\tau-\tau=O_p\left(\frac{1}{n\Delta_n (\gamma_2-\gamma_1)^2}\right).
\end{equation}
\end{theorem}
\begin{proof}
We use the same framework of the proof of the Proposition 3 in Bai (1994), Section 4, therefore we omit the details.

We choose a $\delta>0$ such that $\tau\in (\delta,1-\delta)$. Since $\hat k/n$ is consistent for $\tau$, for every $\varepsilon>0$, $Pr\{\hat k/n\not\in (\delta,1-\delta)\}<\varepsilon$ when $n$ is large. In order to prove \eqref{imrate} it is sufficient to show that $Pr\{|\hat\tau-\tau|>M(n\Delta_n\gamma_n^2)^{-1}\}$ is small when $n$ and $M$ are large, where $\gamma_n=\gamma_2-\gamma_1$. We are interested to study the behavior of $V_k$ for $n\delta\leq k \leq n(1-\delta)$, $0<\delta<1$. We define for any $M>0$ the set $D_{n,M}=\{k:n\delta\leq k \leq n(1-\delta),|k-k_0|>M\Delta_n^{-1}\gamma_n^{-2}\}$. Then we have that
$$Pr\{|\hat\tau-\tau|>M(n\Delta_n\gamma_n^2)^{-1}\}\leq \varepsilon+Pr\{\sup_{k\in D_{n,M}}|V_k|\geq |V_{k_0}|\},$$
for every $\varepsilon>0.$ Thus we study the behavior of $Pr\{\sup_{k\in D_{n,M}}|V_k|\geq |V_{k_0}|\}$. It is possible to prove that
\begin{equation}
\label{int1}
\begin{aligned}
Pr\left\{\sup_{k\in D_{n,M}}|V_k|\geq |V_{k_0}|\right\}\leq&\quad Pr\left\{\sup_{k\in D_{n,M}} V_k - V_{k_0}\geq 0\right\}\\
& +Pr\left\{\sup_{k\in D_{n,M}}V_k+V_{k_0}\leq 0\right\}\\
=&\quad P+Q
\end{aligned}
\end{equation}   
Furthermore
\begin{eqnarray}\label{int2}
 Q
 &\leq& 2Pr\left\{\sup_{k\leq n(1-\delta)}\frac{1}{n-k}\left|\sum_{i=k+1}^n (Y_i-\gamma_2)\right|\geq \frac{1}{4}EV_{k_0}\right\}\\
 &&+2Pr\left\{\sup_{k\geq n\delta}\frac{1}{k}\left|\sum_{i=1}^k (Y_i-\gamma_1)\right|\geq \frac{1}{4}EV_{k_0}\right\}.\notag
\end{eqnarray}
By observing that $\sum_{i=m}^\infty i^{-2}=O(m^{-1})$,  the Haj\'{e}k-Renyi inequality yields
\begin{equation}\label{hr2}
P\left\{\max_{ k\geq  m}\left|\frac{1}{k}\sum_{i=1}^k Z_i\right|>\alpha \right\} \leq \frac{1}{\alpha^2m}\frac{(1-e^{\lambda\Delta_n})e^{-\lambda\Delta_n}}{\sqrt{\Delta_n}},
\end{equation}
where r.v.'s $Z_i$ are defined in the proof of Theorem \ref{teo1}.
The inequality \eqref{hr2} implies that \eqref{int2} tends to zero as $n$ tends to infinity.
Let $b(k)=\sqrt{((k/n)(1-k/n))}, k=1,2,...,n,$ for the first term in the right-hand of \eqref{int1} we have that
\begin{equation}
\label{int3}
\begin{aligned}
P\leq & \quad Pr\left\{\sup_{k\in D_{n,M}}\frac{n}{|k_0-k|}|G(k)|>\frac{\gamma_nC_\tau}{2}\right\}\\
&+Pr\left\{\sup_{k\in D_{n,M}}\frac{n}{|k_0-k|}|H(k)|>\frac{\gamma_nC_\tau}{2}\right\}\\
=&\quad P_1+P_2,
\end{aligned}
\end{equation}
where
\begin{equation}\label{g}
G(k)=b(k_0)\frac{1}{k_0}\sum_{i=1}^{k_0}(Y_i-\gamma_1)-b(k)\frac{1}{k}\sum_{i=1}^{k}(Y_i-\gamma_1)
\end{equation}
\begin{equation}\label{h}
H(k)=b(k)\frac{1}{n-k}\sum_{i=k+1}^{n}(Y_i-\gamma_2)-b(k_0)\frac{1}{n-k_0}\sum_{i=k_0+1}^{n}(Y_i-\gamma_2)
\end{equation}
We prove that $P_1$ tends to zero when $n$ and $M$ are large. Thus
we consider only $k\leq k_0$ or more precisely those values of $k$ such that $n\delta\leq k \leq n\tau-M\Delta_n^{-1}\gamma_n^{-2}$. For $k\geq n \delta$, we have
\begin{equation}\label{modG}
|G(k)|\leq \frac{k_0-k}{n\delta k_0}\left|\sum_{i=1}^{k_0}(Y_i-\gamma_1)\right|+B\frac{k_0-k}{n}\frac{1}{n\delta }\left|\sum_{i=1}^{k}(Y_i-\gamma_1)\right|+\frac{1}{n\delta }\left|\sum_{i=k+1}^{k_0}(Y_i-\gamma_1)\right|,
\end{equation}
where $B\geq 0$ satisfies $|b(k_0)-b(k)|\leq B|k_0-k|/n$.
By means of \eqref{hr}, \eqref{hr2} and \eqref{modG},  we obtain
\begin{eqnarray*}
P_1&\leq& Pr\left\{\frac{1}{n\tau}\left|\sum_{i=1}^{[n\tau]}(Y_i-\gamma_1)\right|>\frac{\delta\gamma_n C_\tau}{6}\right\}\\
&&+Pr\left\{\sup_{1\leq k \leq n}\frac{1}{n}\left|\sum_{i=1}^{k}(Y_i-\gamma_1)\right|>\frac{\delta\gamma_n C_\tau}{6B}\right\}\\
&&+Pr\left\{\sup_{ k \leq n\tau-M\Delta_n^{-1}\gamma_n^{-2}}\frac{1}{n\tau-k}\left|\sum_{i=k+1}^{[n\tau]}(Y_i-\gamma_1)\right|>\frac{\delta\gamma_n C_\tau}{6}\right\}\\
&\leq&\frac{36 D }{(\delta C_\tau)^2 \tau n\Delta_n\gamma_n^2}+\frac{36B^2D}{(\delta C_\tau)^2  n\Delta_n\gamma_n^2}+\frac{36D}{\delta C_\tau^2M}
\end{eqnarray*}
where $D=\frac{(1-e^{-\lambda\Delta_n})e^{-\lambda\Delta_n}}{\Delta_n}\leq\lambda+o(1)$. When $n$ and $M$ are large the last three terms are negligible. Analogously we derive the proof of $P_2$.

 \end{proof}
\section{Asymptotic distributions}\label{dist}

We want to study in this Section the asymptotic distribution of $\hat\tau$ under our limiting framework for small variations of the rate of change of the direction. The case $\lambda_n=\lambda_2-\lambda_1$ equal to a constant is less interesting because when $\lambda_n$ is large the estimate of $k_0$ is quite precise.  

We note that $\lambda_n=\lambda_2-\lambda_1 \to 0$ implies $\gamma_n=\gamma_2-\gamma_1 \to 0$. By adding the condition
\begin{equation}\label{con}
\lambda_n\to 0,\quad \frac{\sqrt{n\Delta_n}\gamma_n}{\sqrt{\log n}}\to \infty,
\end{equation}
 the consistency of $\hat\tau$ follows immediately either from Theorem \ref{teo1} or Theorem \ref{consim}. In order to obtain the main result of this Section, it is useful to observe that
\begin{equation}
\hat k_0=\arg\max_kV_k^2=\arg\max_kn\Delta_n(V_k^2-V_{k_0}^2)
\end{equation}
and to define a two-sided Brownian motion $W(v)$ in the following manner
\begin{equation}
W(u)=
\begin{cases}
W_1(-u),& u<0 \\
W_2(u), &u\geq0
\end{cases}
\end{equation} 
where $W_1,W_2$ are two independent Brownian motions. Now we present the following convergence in distribution result.
\begin{theorem}
Under assumption \eqref{con}, for $n\Delta_n\to \infty,\Delta_n\to 0$ as $n\to \infty$, we have that
\begin{equation}
\frac{n\Delta_n \gamma_n^2(\hat\tau-\tau)}{\tilde\lambda}\stackrel{d}{\to}
\arg\max_v\left\{W(v)-\frac{|v|}{2}\right\},
\end{equation}
where $W(v)$ is a two-sided Brownian motion and $\tilde\lambda$ is any consistent estimator for $\lambda_1$ or $\lambda_2$.
\end{theorem}
\begin{proof}
The proof follows the same steps in Bai (1994), Theorem 1, hence we only sketch the parts of the proof that differ.
We consider only $v\leq 0$ because of symmetry.
Let $K_n(v)=\{k:k=[k_0+v\Delta_n^{-1}\gamma_n^{-2}], -M\leq v \leq 0, M>0\}$ and
\begin{equation}
\Lambda_n(v)=n\Delta_n(V_k^2-V_{k_0}^2)
\end{equation}
with $k\in K_n(v).$ We note that
\begin{eqnarray}\label{con1}
n\Delta_n(V_k^2-V_{k_0}^2)&=&2n\Delta_nEV_{k_0}(V_k-V_{k_0})\notag\\
&&+2n\Delta_n(V_{k_0}-EV_{k_0})(V_k-V_{k_0})\notag\\
&&+n\Delta_n(V_k-V_{k_0})^2
\end{eqnarray}
The last two terms in \eqref{con1} are negligible on $K_n(v)$. Since $\sqrt{n\Delta_n}(V_{k_0}-EV_{k_0})$ is bounded by \eqref{boundin}, we have to show that $\sqrt{n\Delta_n}|V_k-V_{k_0}|$ is bounded. In particular, we can write 
$$\sqrt{n\Delta_n}|V_k-V_{k_0}|\leq \sqrt{n\Delta_n}|G(k)+H(k)|+\sqrt{n\Delta_n}|EV_k-EV_{k_0}|,$$
where $G(k)$ and $H(k)$ are defined respectively in \eqref{g} and \eqref{h}. The upper bound \eqref{modG} is $o_p(1)$, because the first term is such that
\begin{equation}\label{smallo}
\begin{aligned}
\sqrt{n\Delta_n}\frac{k_0-k}{n\delta k_0}\left|\sum_{i=1}^{k_0}(Y_i-\gamma_1)\right|&
\leq \frac{M}{\delta\tau n\Delta_n\gamma_n^2}\left|\sum_{i=1}^{k_0}\sqrt{\Delta_n}(Y_i-\gamma_1)\right|\\
&=\frac{O_p(1)}{n\Delta_n\gamma_n^2}=o_p(1),
\end{aligned}
\end{equation}
similarly for the second term and for the third term we apply the invariance principle \eqref{inv}.
Now we explicit the limiting distribution for
\begin{equation}
2n\Delta_nEV_{k_0}(V_k-V_{k_0})=2\sqrt{\tau(1-\tau)}n\Delta_n\gamma_n(V_{[k_0+v\Delta_n^{-1}\lambda_n^{-2}]}-V_{k_0}).
\end{equation}
For simplicity we shall assume that $k_0+v\Delta_n^{-1} \gamma_n^{-2}$ and $v\Delta_n^{-1} \gamma_n^{-2}$ are integers. We observe that
\begin{equation}
n\Delta_n\gamma_n(V_k-V_{k_0})=n\Delta_n\gamma_n(G(k)+H(k))-n\Delta_n\gamma_n(EV_{k_0}-EV_k),
\end{equation}
where $G(k),H(k)$ are defined in the expressions \eqref{g}, \eqref{h}. We can rewrite $G(k)$ as follws
\begin{equation}
\label{gbis}
\begin{aligned}
G(k) =&  \quad b(k_0)\frac{k-k_0}{kk_0}\sum_{i=1}^{k_0}(Y_i-\gamma_1)+\frac{b(k_0)-b(k)}{k}\sum_{i=1}^k(Y_i-\gamma_1)\\
&+b(k_0)\frac{1}{k}\sum_{i=k+1}^{k_0}(Y_i-\gamma_1).
\end{aligned}
\end{equation}
By the same arguments used to prove \eqref{smallo} we can show that the first two terms in \eqref{gbis} multiplied by $n\Delta_n\gamma_n$ are negligible on $K_n(M)$. Furthermore $b(k_0)=\sqrt{\tau(1-\tau)}$ and $n/k\to 1/\tau$ for $k\in K_n(M)$, then we get that
\begin{eqnarray}
n\Delta_n\gamma_n G(k_0&+&v\Delta_n^{-1}\gamma_n^{-2})=n\Delta_n\gamma_nb(k_0)\frac{1}{k}\sum_{i=k+1}^{k_0}(Y_i-\gamma_1)+o_p(1)\notag\\
&=&b(k_0)\frac{n}{k}\left\{\gamma_n\sqrt{\Delta_n}\sum_{i=k+1}^{k_0}\sqrt{\Delta_n}(Y_i-\gamma_1)\right\}+o_p(1)\notag\\
&=&b(k_0)\frac{n}{k}\left\{\gamma_n\sqrt{\Delta_n}\sum_{i=1}^{|v|\Delta_n^{-1}\gamma_n^{-2}}\sqrt{\Delta_n}(Y_{i+k}-\gamma_1)\right\}+o_p(1)\notag\\
&\stackrel{d}{\to}&\frac{\sqrt{(1-\tau)\tau}}{\tau} \sqrt{\lambda_1}W_1(-v)
\end{eqnarray}
where in the last step we have used the invariance principle \eqref{inv}. Analogously we show that
\begin{equation}
n\Delta_n\gamma_n H(k_0+v\Delta_n^{-1}\gamma_n^{-2})
\stackrel{d}{\to}\frac{\sqrt{(1-\tau)\tau}}{1-\tau} \sqrt{\lambda_1}W_1(-v).
\end{equation}
Since
\begin{equation}
n\Delta_n\gamma_n(EV_{k_0}-EV_k)\to \frac{|v|}{1\sqrt{\tau(1-\tau)}}
\end{equation}
we obtain that
\begin{equation}
\Lambda_n(v)\stackrel{d}{\to}2\left\{ \sqrt{\lambda_1}W_1(-v)-\frac{|v|}{2}\right\}.
\end{equation}
In the same way, for $v>0$, we can prove that
\begin{equation}
\Lambda_n(v)\stackrel{d}{\to}2\left\{ \sqrt{\lambda_1}W_2(v)-\frac{|v|}{2}\right\}.
\end{equation}
By applying the continuous mapping theorem and Theorem \ref{consim}.
\begin{equation}
\frac{n\Delta_n \gamma_n^2(\hat\tau-\tau)}{\hat\lambda}\stackrel{d}{\to}\frac{1}{\lambda_1}
\arg\max_v\Lambda_n(v).
\end{equation}
Since $aW(v)\stackrel{d}{=}W(a^2 v),a\in \mathbb{R}$, a change in variable transforms
$\arg\max_v\Lambda_n(v)$ into
  $\lambda_1\arg \max_v\left\{W(v)-\frac{|v|}{2}\right\},$ which concludes the proof. 
\end{proof}

%\begin{remark}
%Let $\hat\gamma_n=\hat\gamma_2-\hat\gamma_1$ be a consistent estimator of $\gamma_n$. A $100(1-\alpha)\%$ confidence interval is given by
%$$\left[\hat k-\left[c_\alpha\frac{\hat\gamma_n^2}{\hat\lambda}\right]-1,\hat k+\left[c_\alpha\frac{\hat\gamma_n^2}{\hat\lambda}\right]+1\right],$$
%where $c_\alpha$ is the $(1-\alpha/2)$th quantile of $\arg\max_v\left\{W(v)-\frac{|v|}{2}\right\}$. 
%\end{remark}
%The confidence intervals so-obtained require the estimate of two parameters in order to derive the critical values. Antoch and Hu\v{s}kov\'{a} (1995), performed a bootstrapping procedure to approximate the quantile of $\arg\max_v\left\{W(v)-\frac{|v|}{2}\right\}$ from which we can obtain an approximated critical value $c_\alpha$.

Using the consistency result, we are able to obtain the asymptotic distributions for the estimators $\hat \lambda_1,\hat\lambda_2$, defined in \eqref{estlam}.
\begin{theorem}
Under the assumption \eqref{con} we have that
\begin{equation}\label{estcon}
\sqrt{n\Delta_n}\left( \begin{array}{c} 
\hat\lambda_1 \\ 
\hat\lambda_2 \\  
\end{array} \right) \stackrel{d}{\to}N\left( 0,\Sigma \right),
\end{equation} 
where
\begin{equation}
\Sigma=\left( \begin{array}{cc} 
\tau^{-1}\lambda_1& 0\\ 
0 & (1-\tau)^{-1}\lambda_2 \\  
\end{array} \right).
\end{equation}
\end{theorem}
\begin{proof}
We start noticing that
\begin{eqnarray}\label{conv}
&&\sqrt{n\Delta_n}(\hat\gamma_1(\hat k)-\hat\gamma_1(k_0))\\
&&=\sqrt{n\Delta_n}\left(\frac{1}{\hat k}\sum_{i=1}^{\hat k}Y_i-\frac{1}{ k_0}\sum_{i=1}^{k_0}Y_i\right)\notag\\
&& =\mathbf{1}_{\{\hat k\leq k_0\}}\left(\sqrt{n\Delta_n} \frac{k_0-\hat k}{k_0\hat k}\sum_{i=1}^{k_0}\left( Y_i-\gamma_1\right)-\sqrt{n\Delta_n}\frac{1}{\hat k}\sum_{i=\hat k}^{k_0}\left( Y_i-\gamma_1\right)\right)\notag\\
&&\quad+ \mathbf{1}_{\{\hat k >k_0\}}\Biggl(\sqrt{n\Delta_n}
\frac{k_0-\hat k}{k_0\hat k}\sum_{i=1}^{k_0}\left(
Y_i-\gamma_1\right)+\sqrt{n\Delta_n}\frac{1}{\hat k}\sum_{i=
k_0}^{\hat k}\left( Y_i-\gamma_2\right)\notag\\
&& \quad+\sqrt{n\Delta_n}
\gamma_n\frac{\hat k-k_0}{\hat k}\Biggr).\notag
\end{eqnarray}
Since $k_0=[\tau n]$, $\hat k=k_0+O_p(\Delta_n^{-1}\gamma_n^{-2})$,
and $n\Delta_n\gamma_n^2 \to \infty$, we have that \eqref{conv}
is $(\sqrt{n}\gamma_n)^{-1}O_p(1)$, which converges to zero in
probability. Then $\hat\lambda_1(\hat k)=-\frac{1}{\Delta_n}\log(1-\hat\gamma_1(\hat k)\Delta_n)$ and $\hat\lambda_1( k_0)=-\frac{1}{\Delta_n}\log(1-\hat\gamma_1(k_0)\Delta_n)$
have the same limiting distribution. Obviously the same result holds for $\hat\lambda_2$. By Theorem 4.1 in Iacus and Yoshida (2007), the convergence result \eqref{estcon} follows.

\end{proof}
\section{Application to real data}\label{appl}
In this section we consider an application of our model to two well known real data sets.
The first data set is about the Dow-Jones industrial average and the second one is the IBM stock prices.
In both cases, the data mesh $\Delta_n$ is not close to zero, hence the asymptotics of our set up does not hold. Nevertheless, our findings seems to confirm the results of previous analyses.

\subsection{Dow-Jones data}
This data set contains the
weekly closings of the Dow-Jones industrial average in the period July 1971 - Aug 1974.
These data have been proposed by Hsu (1977, 1979) and used by many other authors to test change point estimators. There are 162 data and the main evidence found by several authors is that a change in the variance occurred at point 89th which corresponds to the third week of March 1973. 
Instead of working on the values we transform the data into returns as usual $X(t_i) = (W(t_i)-W(t_{i-1}))/W(t_{i-1})$, $i=1, \ldots, n$ with $W$ the series of Dow-Jones closings and $X$ the returns. We assume that $X$ follows a telegraph process.
\begin{figure}[t]
\centering{\includegraphics{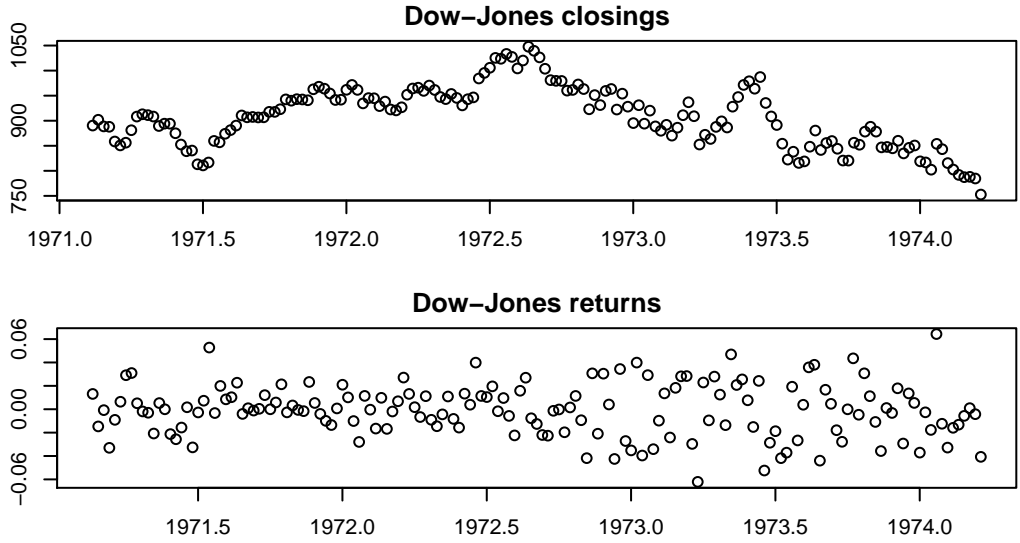}}
\caption{Weekly closings of the Dow-Jones industrial average July 1971 - Aug 1974 (up) and correspoind returns (down).}
\label{fig1}
\end{figure}
In this application, the data are not sampled at high frequency, i.e. $\Delta_n$ is not close to zero, hence we test our estimator of the change point even if the asymptotics is not realized.
Further, and for the same reason, we cannot assume as known the velocity $v$ of the process hence we first estimate $v$ by the average of the rescaled increments. i.e.
$$
\hat v_n = \frac{1}{n} \sum_{i=1}^n \frac{\eta_i}{\Delta_n}
$$
This is a consistent estimator of $v$, hence we construct the estimator of $\gamma$ as follows
$$
\hat\gamma = \frac{1}{n}\sum_{i=1}^n Y_i \qquad\text{where}\qquad Y_i =  \frac{\mathbf{1}_{\{|\eta_i|<\hat v_n \Delta_n\}}}{\Delta_n}
$$
With these quantities, we construct the statistics $D_k$ in \eqref{DK} and maximize it. The maximum is reached at $\hat k_0 = 89$ which confirms the evidence in Hsu (1974, 1979). Once we obtained the estimation of the change point, we re-estimate the velocity in both part of the series (before and after point 89th) and the two lambda's. We obtained respectively $v_1 = 0.61$, $\lambda_1 = 48.53$ and $v_2=1.24$ and $\lambda_2 = 34.61$ which confirms the intuition from the graphical inspection of the returns (i.e. in the first period there is a high number of switches but with low velocity which correspond to low variance of the returns; conversely for the second period).
Looking better at the first part of the series, we observe that variance is not stable, so we re-run the procedure and obtained a new change point $\hat k_1 = 27$ around august 1971. Figure \ref{fig2} contains the two change point estimates plotted against the Dow-Jones returns.
\begin{figure}[t]
\centering{\includegraphics{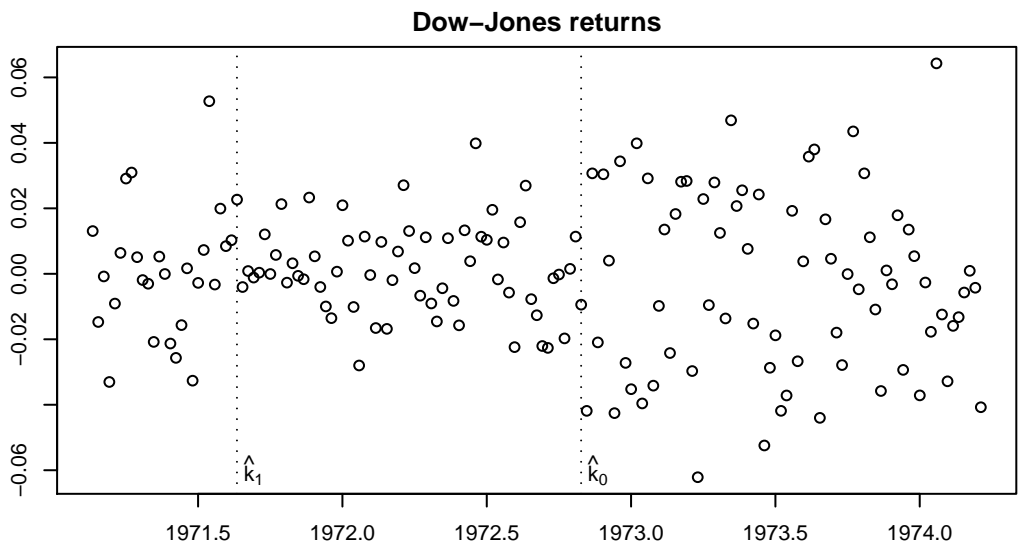}}
\caption{Change point estimates on the returns of the weekly closings of the Dow-Jones industrial average July 1971 - Aug 1974. Major change point estimate $\hat k_0=89$ which corresponds to the 3rd week of March 1973; second change point estimate $\hat k_1 = 27$, August 1971.}
\label{fig2}
\end{figure}
\subsection{IBM stock prices}
This data set contains 369 closing stock prices of the IBM as,set. They have been analyzed in Box and Jenkins (1970)  and further by  Wichern {\it et al.} (1976) in order to discover change points. Box and Jenkins (1970) fitted an ARIMA(0,1,1) on the first order difference and discover heteroschedasticity; Wichern {\it et al.} (1976) fitted an AR(1) model on the first differences of the logarithms. We consider instead the returns as in previous example and apply the same sequential procedure. Data are reported in Figure \ref{fig3} along with a couple of change points discovered by our estimates.
The first change point was found at point $\hat k_0 = 235$ which confirms the findings of Wichern {\it et al.} (1976). We further discovered another change point at time index $\hat k_1 = 18$ on the time series on the left to $\hat k_0$ and  a second change point on the right-hand series at time $\hat k_2 = 309$. 
\begin{figure}[t]
\centering{\includegraphics{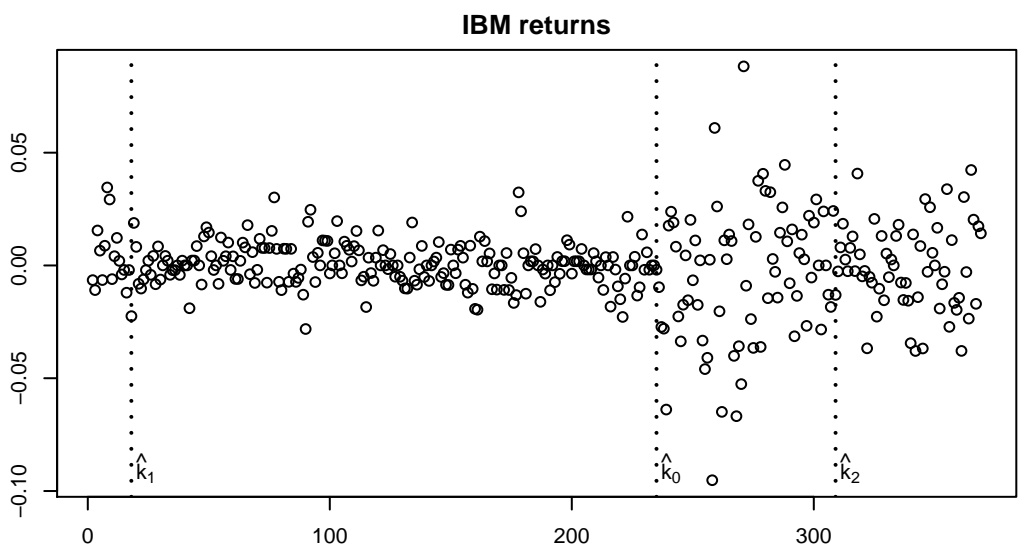}}
\caption{Return of the IBM stock closings (see e.g. Box and Jenkins, 1970). The major change point occurs at index $\hat k_0 = 235$, the other two at $\hat k_1 = 18$ and $\hat k_2=309$.}
\label{fig3}
\end{figure}

\end{document}